\documentclass[11pt]{article}

\usepackage{amsmath,amssymb,amsthm}
\usepackage{enumitem}
\usepackage[margin=1.15in]{geometry}
\usepackage{authblk}

\newtheorem{theorem}{Theorem}[section]
\newtheorem{proposition}[theorem]{Proposition}
\newtheorem{lemma}[theorem]{Lemma}
\newtheorem{corollary}[theorem]{Corollary}
\newtheorem{example}[theorem]{Example}
\newtheorem{remark}[theorem]{Remark}

\newtheorem*{theorem*}{Theorem}

\theoremstyle{definition}
\newtheorem{definition}[theorem]{Definition}

\newcommand{\R}{\mathbb{R}}
\newcommand{\id}{\operatorname{id}}

\title{Fusion rules from the Norton inequality}
\author{Alonso Castillo-Ramirez\footnote{Email: alonso.castillor@academicos.udg.mx}}
\affil{Universidad de Guadalajara, Centro Universitario de Ciencias Exactas e Ingenier\'ias, Departamento de Matem\'aticas. Guadalajara, M\'exico.}
\date{}

\begin{document}

\maketitle

\begin{abstract}
The Norton inequality is one of the fundamental axioms in the theory of Majorana and axial algebras, yet its precise structural consequences have remained only partially understood. In this paper, we show that the Norton inequality alone forces the $0$- and $1$-eigenspace fusion rules for arbitrary idempotents in a commutative real algebra $A$ equipped with a Frobenius form. More precisely, if the Frobenius form is nondegenerate (as in Majorana algebras), we prove that the eigenspaces $A_0(e)$ and $A_1(e)$ of an arbitrary idempotent $e \in A$ are subalgebras and annihilate one another:
\[
        A_0(e)A_1(e)=\{0\},
\]
while in the degenerate case the corresponding inclusions hold modulo the radical of the Frobenius form. This answers a question of T. M. Mudziiri Shumba and S. Shpectorov concerning the closure of the $0$-eigenspace $A_0(e)$.
\end{abstract}

\section{Introduction}

Axial algebras are commutative non-associative algebras generated by
distinguished idempotents, called axes, whose adjoint actions satisfy prescribed
fusion laws. The motivating examples include the Griess algebra, Majorana
algebras, and Matsuo algebras \cite{HRS, MSsurvey}. A
recurring feature in these examples is the presence of a Frobenius form, namely
a symmetric bilinear form \((\cdot,\cdot)\) satisfying
\[
        (xy,z)=(x,yz), \qquad \forall x,y,z \in A.
\]

In \cite{Norton}, Norton found several fundamental identities of the Griess
algebra \(V_M\), including the inequality now bearing his name,
\[
        (x^2,y^2)\geq (xy,xy), \qquad \forall x,y\in V_M.
\]
Although this inequality is the axiom M2 of Majorana algebras \cite{Ivanov}, its precise structural role has remained unclear. For example, in \cite[p. 2]{MA12}, the authors observe that the Norton inequality has never been used in the development of Majorana theory. Determining which fusion rules are consequences of the Norton inequality alone, rather than of the stronger axioms usually imposed in Majorana and axial algebra theory, is therefore a natural structural problem.

From a complementary geometric viewpoint, Fox \cite{Fox} interprets the
Norton inequality for metrized commutative algebras as nonnegative sectional
nonassociativity and, in the Euclidean case, studies the spectrum of the adjoint operator of idempotents. 

In this paper, we prove that the fusion rules of the $0$- and $1$-eigenspaces of an arbitrary idempotent in a real commutative algebra with a Frobenius form are forced by the Norton inequality. In this general setting, for an arbitrary idempotent \(e\in A\), we relate the eigenspaces \(A_0(e)\) and \(A_1(e)\) to the radical \(A^\perp\) of the Frobenius form. If the Frobenius form is nondegenerate, as in the case of Majorana algebras, then both \(A_0(e)\) and \(A_1(e)\) are subalgebras of \(A\), and
\[
        A_0(e)A_1(e)=\{0\}.
\]
In the degenerate case, the corresponding inclusions hold modulo \(A^\perp\), and the mixed fusion rule takes the particularly simple form
\[
        A_0(e)A_1(e)\subseteq A^\perp.
\]

As a consequence, our results answer a question of T. M. Mudziiri Shumba and S. Shpectorov concerning the closure of the $0$-eigenspace of idempotents in algebras with a Frobenius form that satisfy the Norton inequality (see \cite{Auto} and \cite[Problem~5.5]{GSQuestions}).

The fact that \(A_0(e)\) and \(A_1(e)\) are subalgebras for every idempotent \(e\in A\) was proved by Conway \cite{Conway}, in the context of the Griess algebra, and by the author \cite{Castillo}, in the context of Majorana algebras, under the strengthened axiom \(M2'\), namely
\[
        (x^2,y^2)\geq (xy,xy), \quad \forall x, y \in A,
\]
with equality if and only if the adjoint operators \(L_x\) and \(L_y\) commute. This was a key ingredient in the study of associative subalgebras of Majorana algebras, and in particular of the Griess algebra; see \cite{Meyer}. Since the role of M2' in \cite{Castillo} is to obtain precisely the closure of the eigenspaces $A_0(e)$ and $A_1(e)$, the results of the present paper show that M2' may be replaced there by the ordinary Norton inequality axiom M2. In contrast with the proofs in \cite{Castillo, Conway}, the proofs of this paper only use the Norton inequality itself and are independent of any prescribed axial fusion law. They are based on a \emph{Norton form} whose positive semidefiniteness is implied by the Norton inequality. Thus our results apply to arbitrary idempotents in any commutative real algebra with a Frobenius form satisfying the Norton inequality, and the usual fusion-law language appears only in our conclusions.


\section{Preliminaries}

Throughout this paper, an algebra means a commutative, not necessarily
associative, algebra over \(\R\).

\begin{definition}
Let \(A\) be an algebra. A symmetric bilinear form \((\cdot,\cdot)\) on \(A\)
is called a \emph{Frobenius form}, or an \emph{associative form}, if
\[
        (xy,z)=(x,yz)
\]
for all \(x,y,z\in A\).
\end{definition}

Equivalently, for every \(x\in A\), the adjoint operator, defined by
\[
        L_x:A\to A,\qquad L_x(y)=xy, \quad \forall y \in A, 
\]
is self-adjoint with respect to \((\cdot,\cdot)\).

\begin{definition}
Let \(A\) be an algebra equipped with a Frobenius form. We say that \(A\)
satisfies the \emph{Norton inequality} if
\[
        (x^2,y^2)\geq (xy,xy)
\]
for all \(x,y\in A\).
\end{definition}

For any $x \in A$ and \(\lambda\in\R\), denote the $\lambda$-eigenspace of the adjoint operator \(L_x\) by
\[
        A_\lambda(x)=\{v \in A: xv=\lambda v\}.
\]

We shall use the following elementary results.

\begin{lemma}\label{le1}
Let \(B\) be a positive semidefinite symmetric bilinear form on a real vector
space \(V\). If \(B(v,v)=0\), then \(B(v,w)=0\) for all \(w\in V\).
\end{lemma}

\begin{proof}
For any \(w\in V\) and \(t\in\R\),
\[
        0\leq B(v+tw,v+tw)=2tB(v,w)+t^2B(w,w).
\]
Since $0$ is the minimum of the quadratic polynomial, its derivative at $t=0$ must vanish. Hence \(B(v,w)=0\).
\end{proof}

Alternatively, the previous result follows from the Cauchy–Schwarz for positive semidefinite symmetric bilinear real forms: $|B(v,w)|^2 \leq B(v,v) B(w,w)$ for all $v,w \in V$. 

\begin{lemma}\label{le2}
Let $A$ be an algebra with a Frobenius form. For any
$x \in A$, eigenspaces of $L_x$ corresponding to distinct eigenvalues are
orthogonal.
\end{lemma}
\begin{proof}
Let $v \in A_\lambda(x)$ and let $w \in A_\mu(x)$ with $\lambda \neq \mu$. Then,
\[ \lambda (v,w) = (\lambda v, w) = (xv , w) = (v, xw) = (v, \mu w) = \mu (v,w).  \]
Therefore, $(v,w) = 0$. 
\end{proof}


\section{The Norton inequality and fusion rules}

The main arguments of this paper are based on the definition of an auxiliary symmetric bilinear form given for each $u \in A$. 

\begin{definition}\label{Norton-form}
Let $A$ be an algebra with a Frobenius form. For each fixed $u \in A$, the
\emph{Norton form} associated with $u$ is the symmetric bilinear form $B_u$ on
$A$ defined by
\[
        B_u(x,y)
        :=\bigl((L_{u^2}-L_u^2)x,y\bigr)
        =(u^2,xy)-(ux,uy),
        \qquad \forall x,y\in A.
\]
In particular, if $e \in A$ is an idempotent, then
\[
        B_e(x,y)=\bigl((L_e-L_e^2)x,y\bigr).
\]
\end{definition}

As we note in the following lemma, the Norton inequality implies that Norton forms are always positive semidefinite; this fact has been noted before in  \cite[Eq. 6.1]{Fox}.

\begin{lemma}\label{le3}
If an algebra \(A\) with a Frobenius form satisfies the Norton inequality, then the Norton form \(B_u\) is positive semidefinite for all $u \in A$.
\end{lemma}
\begin{proof}
This follows directly from the definition of the Norton inequality:
\[        B_u(x,x)=(u^2,x^2)-(ux,ux)\geq 0 \]
for all \(x\in A\). Thus \(B_u\) is positive semidefinite.
\end{proof}

In the context of axial algebras, Frobenius forms are not necessarily nondegenerate, and their radicals play an important role in the theory of their structure (see \cite{KMS}). We denote the radical of the Frobenius form of $A$ by 
\[
        A^{\perp}= \{u\in A:(u,x)=0, \forall x \in A \}.
\]
The Frobenius property implies that $A^\perp$ is an ideal of $A$: if
$r\in A^\perp$ and $a,x\in A$, then $(ar,x)=(r,ax)=0$.

For subspaces $S$ and $T$ of $A$, define $ST := \langle st : s \in S, t \in T \rangle$. The following theorem is our key result.

\begin{theorem}\label{th-main}
Let \(A\) be a commutative algebra over \(\R\), equipped with a Frobenius form that satisfies the Norton inequality. 
For every idempotent \(e\in A\) and every \(\lambda \in\{0,1\}\),
\[
        (L_e - \lambda \id_A) \left( A_\lambda(e)A_\lambda(e) \right)
        \subseteq A^\perp
\]
and the mixed product satisfies the stronger inclusion
\[
        A_0(e)A_1(e) \subseteq A^\perp .
\]
\end{theorem}

\begin{proof}
Let \( x \in A_\lambda(e)\) and consider the Norton form $B_x$, which is positive semidefinite by Lemma \ref{le3}. As $e$ is idempotent, we have
\begin{align*}
 B_x(e,e) & = (x^2, e^2)-(xe, xe) \\
  & = (x^2,e) - (\lambda x, \lambda x)  \\
  & = (x,xe) - \lambda^{2}(x,x) \\
  & = \lambda(x,x) - \lambda^{2}(x,x) \\
  & =  (\lambda - \lambda^2) (x,x) = 0 ,
\end{align*}
since $\lambda \in \{0,1\}$. By Lemma \ref{le1}, $B_x(e,z) = 0$ for all $z \in A$. Hence,
\begin{align*}
 0 & = B_x(e,z) \\
 &  = (x^2, ez) - (xe,xz) \\
 & = (ex^2, z) - (\lambda x, xz) \\
 &  = (ex^2 -\lambda x^2, z), 
 \end{align*}
for all $z \in A$. This shows that
\[ (L_e - \lambda \operatorname{id}_A)(x^2) \in A^{\perp}, \quad \forall x \in A_\lambda(e).\]
Now, for any $x,y \in A_\lambda(e)$, we use the identity
\[ xy = \frac{1}{2}(x+y)^2 - \frac{1}{2}x^2 - \frac{1}{2}y^2.  \]
Therefore,  
\[ (L_e - \lambda \operatorname{id}_A)(xy) = \frac{1}{2}(L_e - \lambda \id_A)((x+y)^2) - \frac{1}{2}(L_e - \lambda \id_A)(x^2) - \frac{1}{2}(L_e - \lambda \id_A)(y^2) \in A^\perp. \]  
We now prove the mixed case. Let $x \in A_0(e)$ and $y \in A_1(e)$, and set $u:=x+y$. Since $eu = y$, we have
\begin{align*}
 B_u(e,e) & = (u^2, e^2) - (ue, ue) \\
& = (u^2,e) - (y,y) \\
& = (u, eu) - (y,y) \\
& = (x+y,y) - (y,y) \\
& = (x,y) =0, 
\end{align*}
where the last equality follows by Lemma \ref{le2}. Now, by Lemma \ref{le1}, $B_u(e,z)=0$ for all $z \in A$. This means that
\begin{align*}
 0= B_u(e,z) & = (eu^2, z) - (ue, uz) \\
& = (e(x+y)^2, z) - (y, (x+y)z) \\
& = (e(x^2 + 2xy + y^2),z) - (xy,z) - (y^2,z). 
\end{align*}
Therefore, 
\[ ex^2 + ey^2 + 2 e(xy) - xy - y^2 \in A^\perp. \]
By the first part of the proof, $ex^2 \in A^\perp$ and
$ey^2-y^2\in A^\perp$, so
\begin{equation}\label{eq:mixed-mod-radical}
        (2L_e-\id_A)(xy)\in A^\perp.
\end{equation}
Set $z:=xy$. Now \eqref{eq:mixed-mod-radical} is equivalent to 
\[ L_e z + A^\perp = \frac{1}{2} z + A^\perp, \]
and since $A^\perp$ is an ideal, it is $L_e$-invariant, so we have
\begin{align*}
L_e \left( L_e z + A^\perp \right) & = L_e \left( \frac{1}{2} z + A^\perp \right) \\
L_e^2 z + A^\perp & = \frac{1}{2} \left( L_e z + A^\perp \right) \\
L_e^2 z + A^\perp & = \frac{1}{4} z + A^\perp. 
\end{align*}
Consequently,
\begin{equation}\label{eq:Be-mixed}
        B_e(z,w) = \bigl((L_e-L_e^2)z, w\bigr) =\frac14(z,w), \qquad \forall w\in A.
\end{equation}
Moreover,
\[
        (x^2,y^2)=(x^2,ey^2)=(ex^2,y^2)=0,
\]
where the first equality holds because $ey^2-y^2\in A^\perp$. Applying the
Norton inequality to $x$ and $y$ now gives
\[
        (z,z)\leq (x^2,y^2)=0.
\]
On the other hand, the positive semidefiniteness of $B_e$ and
\eqref{eq:Be-mixed} give
\[
        0\leq B_e(z,z)=\frac14(z,z).
\]
Therefore, $(z,z)=0$ and $B_e(z,z)=0$. Lemma \ref{le1} applied to $B_e$ yields
$B_e(z,w)=0$ for every $w\in A$. By \eqref{eq:Be-mixed}, $(z,w)=0$ for every $w\in A$,
and hence $z=xy\in A^\perp$.
\end{proof}

When the Frobenius form is nondegenerate, Theorem \ref{th-main} gives us the following consequence.

\begin{corollary}\label{cor1}
Let \(A\) be a commutative algebra over \(\R\) equipped with a nondegenerate Frobenius form that satisfies the Norton inequality. If \(e\in A\) is an idempotent, then both \(A_0(e)\) and \(A_1(e)\) are
subalgebras of \(A\), and they annihilate one another:
\[
        A_0(e)A_1(e)= \{0\}.
\]
\end{corollary}

\begin{proof}
Since the Frobenius form is nondegenerate, \(A^\perp=\{0\}\). By Theorem
\ref{th-main}, \(e(xy)=\lambda xy\) for every \(x,y\in A_\lambda(e)\) and
\(\lambda\in\{0,1\}\). Thus \(xy\in A_\lambda(e)\), so \(A_\lambda(e)\) is a
subalgebra of \(A\).

The mixed fusion rule follows directly from the second inclusion in Theorem
\ref{th-main}.
\end{proof}

This nondegenerate case includes the situations when the Frobenius form is positive definite, such as in Majorana algebras and the Griess algebra \cite{Ivanov,Norton}. It also
applies to Matsuo algebras whose canonical Frobenius form is nondegenerate and satisfies the Norton inequality; see, for example \cite{MSsurvey}.

For primitive axial algebras, the nondegeneracy hypothesis can be expressed in
the intrinsic language of the axial radical. Recall that if \((A,X)\) is a
primitive axial algebra, its axial radical \(R(A,X)\) is the unique maximal
ideal disjoint from the set \(X\) of generating axes \cite{KMS}.

\begin{corollary}
Let \((A,X)\) be a primitive axial algebra over \(\R\) equipped with a Frobenius form satisfying the Norton inequality and $(a,a)\neq 0$, for every $a\in X$. If \(R(A,X)= \{0\}\), then, for every idempotent \(e\in A\), both \(A_0(e)\) and \(A_1(e)\) are subalgebras
of \(A\), and
\[
        A_0(e)A_1(e)= \{0 \}.
\]
\end{corollary}

\begin{proof}
By \cite[Theorem~3.3]{KMS}, if a primitive axial algebra admits a Frobenius
form which is nonzero on every generating axis, then the axial radical
\(R(A,X)\)  coincides with the radical \(A^\perp\) of the
Frobenius form. Hence \(A^\perp= \{0\}\), and the result follows from Corollary
\ref{cor1}.
\end{proof} 

A useful simplification of Theorem \ref{th-main} occurs when the eigenvalue $\lambda \in \{0,1\}$ of \(L_e\) is a simple root of its minimal polynomial.

\begin{proposition}
Let \(A\) be a finite-dimensional commutative algebra over \(\R\) equipped with a Frobenius form that satisfies the Norton inequality. Let $e \in A$ be an idempotent such that $\lambda \in \{0,1\}$ is a simple root of the minimal polynomial   of $L_e$. Then, 
\[         A_\lambda(e)^2\subseteq A_\lambda(e)+A^\perp. \]
\end{proposition}
\begin{proof}
Let \(\mu(x) \in \mathbb{R}[x]\) be the minimal polynomial of \(L_e\). Since \(\lambda\) is a simple root of \(\mu(x)\), then 
\[        \mu(x) = (x - \lambda) q(x), \]
where $x-\lambda$ and $q(x)$ are coprime. By \cite[Lemma 4.4]{Cooperstein},  
\begin{equation}\label{direct-sum}
A = \ker(L_e - \lambda \id_A) \oplus \ker( q(L_e)) = A_\lambda(e) \oplus \ker( q(L_e)) .  
\end{equation}

By B\'ezout's identity, there exist $a(x), b(x) \in \mathbb{R}[x]$ such that 
\[1 = a(x) (x - \lambda)  + b(x)q(x) .\] 
Now, $P := b(L_e)q(L_e)$ is the projection onto $A_\lambda(e)$ associated with the decomposition (\ref{direct-sum}), and it satisfies that 
\[ \id_A - P = a(L_e)(L_e - \lambda \id_A). \] 
By Theorem \ref{th-main}, for any $x, y \in A_\lambda(e)$, 
\[ (L_e - \lambda \id_A)(xy) \in A^{\perp}. \]
Therefore, since $A^\perp$ is $L_e$-invariant
\[ xy - P(xy) = a(L_e)(L_e - \lambda \id_A)(xy) \in A^\perp. \]
Since $P(xy) \in A_\lambda(e)$, we conclude that
\[ xy \in A_\lambda(e) + A^\perp, \qquad \forall x,y \in A_\lambda(e). \]
\end{proof}

\begin{corollary}
With the previous notation, if $e \in A$ is a \emph{semisimple} idempotent (i.e., the adjoint operator \(L_e\) is diagonalizable), then for every \(\lambda\in\{0,1\}\),
\[       A_\lambda(e)^2\subseteq A_\lambda(e)+A^\perp. \]
\end{corollary}

We finish the paper with an example of an axial algebra with a degenerate Frobenius form satisfying the Norton inequality in which the $0$-eigenspace of one of its axes is not a subalgebra of $A$, but first we state a useful remark. 

\begin{remark}\label{rmk}
Let $V$ be a commutative associative algebra with a Frobenius form. Then, $V$ satisfies the Norton inequality, as it actually becomes an equality:
\[ (x^2, y^2) = (x, xy^2) = (x, (xy)y) = (xy,xy), \qquad \forall x,y \in V. \]
\end{remark}

\begin{example}\label{ex}
Let \(A\) be the real commutative algebra with basis $\{a_1,a_2,a_3,r\}$, with $a_1, a_2, a_3$ idempotents and multiplication determined by
 \[ a_1a_2=a_1a_3=0,\qquad a_2a_3=\frac12(a_2+a_3-r),\] 
\[ a_1r = a_2r=a_3r=r,\qquad r^2=0. \]
Then $A$ is an axial algebra with axes $a_1, a_2, a_3$ satisfying the following fusion law:
\[
\begin{array}{c|ccc}
\star & 1 & 0 & \frac12 \\ \hline
1 & 1 & 1 & \frac12 \\
0 & 1 & 0,1 & 1 \\
\frac12 & \frac12 & 1 & 1
\end{array}
\]
Note that the axes are nonprimitive because the $1$-eigenspaces $A_1(a_i) = \langle a_i, r \rangle$, for $i \in \{1,2,3\}$, are two-dimensional. Define a symmetric bilinear form on \(A\) by 
\[ (a_1,a_1)= (a_2,a_2)=(a_3,a_3)=(a_2,a_3)=1, \] 
and all other pairings between basis elements equal to zero. Direct calculations show that this is a Frobenius form with radical 
\[ A^\perp=\langle a_2-a_3,r\rangle. \] 
To check that $A$ satisfies the Norton inequality we can note that the quotient 
\[ A / A^\perp = \left\langle a_1 + A^\perp, \; \frac{1}{2} (a_2 + a_3) + A^\perp \right\rangle  \]
is a two-dimensional associative algebra where the Frobenius form is the usual dot product by Remark \ref{rmk}. Since $A / A^\perp$ satisfies the Norton inequality, then $A$ satisfies the Norton inequality. However, 
\[ A_0(a_1)=\langle a_2,a_3\rangle, \]
 while
 \[ (a_2-a_3)^2 =a_2^2-2a_2a_3+a_3^2 =r. \] 
Since $a_1r=r\neq 0$, we have \(r\notin A_0(a_1)\). Thus \(A_0(a_1)\) is not a subalgebra of $A$.
\end{example}

\section*{Use of artificial intelligence}

The author used ChatGPT 5.5 for language editing, organization, and exploratory discussion of possible strengthenings of the results. In particular, the strengthening of the mixed fusion rule was obtained during this interaction. The author wrote the manuscript independently and takes full responsibility for its content.


\end{document}